\documentclass{article}

\usepackage{PRIMEarxiv}

\usepackage[utf8]{inputenc} 
\usepackage[T1]{fontenc}    
\usepackage{cite}
\usepackage{amsmath,amssymb,amsfonts}
\usepackage{algorithmic}
\usepackage{graphicx}
\usepackage{algorithm,algorithmic}
\usepackage{hyperref}
\usepackage{graphics,color}
\usepackage{amsthm}
\usepackage{graphicx}
\usepackage{mathrsfs, bm, url, times}
\usepackage{latexsym}
\usepackage{subfigure}
\usepackage{algorithmic}
\usepackage{algorithm}
\usepackage{epsfig}

\newcommand{\F}{\mathcal F}
\newcommand{\A}{\hat{\mathcal A}}
\newcommand{\J}{\mathcal J}
\newcommand{\T}{\mathcal T}
\newcommand{\TT}{\mathbb T}
\newcommand{\HS}{\mathcal H}

\newcommand{\Prob}{\mathbb P}
\newcommand{\E}{\mathbb E}

\newtheorem{theorem}{Theorem}
\newtheorem{lemma}{Lemma}
\newtheorem{definition}{Definition}

\newtheorem{remark}{Remark}

\title{On Time-Varying Delayed Stochastic Differential Systems with Non-Markovian Switching Parameters
}

\author{
  Xinyu Wu \\
  Fudan University \\
  \texttt{xywu19@fudan.edu.cn} \\
   \And
  Zidong Wang \\
  Brunel University London \\
  \texttt{Zidong.Wang@brunel.ac.uk} \\
  \And
  Wenlian Lu \\
  Fudan University \\
  \texttt{wenlian@fudan.edu.cn} \\
}

\begin{document}
\maketitle

\begin{abstract}
This paper focuses on time-varying delayed stochastic differential systems with stochastically switching parameters formulated by a unified switching behavior combining a discrete adapted process and a Cox process. Unlike prior studies limited to stationary and ergodic switching scenarios, our research emphasizes non-Markovian, non-stationary, and non-ergodic cases. It arrives at more general results regarding stability analysis with a more rigorous methodology. The theoretical results are validated through numerical examples.
\end{abstract}


\section{Introduction}

Stochastic systems are often characterized through randomness (e.g.~noise typically represented as Brownian motion) and the stochastic dynamic environment composed of subsystems and defined switching behavior. This behavior, which identifies the active subsystem at each instant, is generally controlled by a continuous-time stochastic process \cite{Dav}. In situations involving Brownian motions, these stochastically evolving dynamics can be formulated as stochastic differential systems (SDSs), the solutions to which are referred to as the diffusion process. SDSs have been extensively researched in terms of the existence and uniqueness of the solution, as well as their dynamic properties like stability \cite{Koz, Mao, Sko, Protter, Oks}. Notably, these stochastically evolving systems can have variable structures due to random internal or external changes, leading to the emergence of the Markovian jump system. In this system, parameters switch between certain discrete modes, adhering to a Markov process.

To date, stochastic systems with Markovian jumping parameters have garnered considerable research attention, yielding significant results applicable to various fields like automatic control, signal processing, and fault detection \cite{CKH, CL, GengH-TSMCS-20, GengH-IS-17, LongY-IJSS-14, YuanC-Auto-04}. For instance, the over-the-horizon radar tracking problem was addressed in \cite{GengH-TSMCS-20} using a novel fast-update target and intermittent-update ionosphere model to capture the dynamic evolution of the target and ionosphere. The fault detection problem for nonlinear stochastic systems with mixed time-delays and Markovian switching was resolved in \cite{LongY-IJSS-14}. Despite numerous advancements in Markovian switching systems, there are limited studies on non-Markovian switching systems, especially where system stability is a major concern \cite{CL1}. This gap in knowledge serves as one of our motivations.

In addition to Markovian switching, Poisson switching has also sparked considerable research interest in recent decades due to its robust capability to depict certain stochastically evolving system dynamics. For instance, a switched Poisson process was used in \cite{BhatVN-JQRS-94} to model the point process whose rate change is governed by a random mechanism. The optimal control problem for stochastic dynamic systems involving both Poisson and Markovian switching was addressed in \cite{Antonyuk-JM-20}. Recently, there has been a growing interest in the synthesis/analysis of system dynamics combining Poisson switching and Brownian motions \cite{Bra, Mao2, YL, YZZ}. Also, due to unavoidable time delays in dynamic systems, resulting from limited communication efficiency in networked environments, (time-varying) delayed SDSs have received significant attention \cite{Mao1, Mao3, Moh, Mao4}. Notably, most of the examined stochastic processes (dictating the network topologies) have been confined to stationary and/or ergodic cases. Rarely have non-stationary and non-ergodic scenarios been explored, let alone cases accommodating time-delays. This forms the second motivation for this paper.

Pertaining to the above discussion, in this paper, we explore time-varying delayed stochastic differential systems with stochastically switching parameters (TVDSDS-SSPs), with the existence and uniqueness of the solution to the Cauchy problem derived via a similar framework as in \cite{Mao1}. Unlike the stability analysis in \cite{CKH, CL1} that ensures the stability of each subsystem via switching Lyapunov functions, our approach is based on defining stochastic Lyapunov functionals, whose conditional expectation decreases with the requirement of stable subsystems. Additionally, we extend the acclaimed Halanay inequality for time-varying delayed dynamics \cite{LLC} and present the stability theorems. For practical application, we derive sufficient conditions for the stability of a class of nonlinear stochastically switching systems. 
Notably, when the switching is limited to Markovian jumps, our derived conditions align with the existing results in literature.

The main contributions of this paper are as follows. 1) To the best of our knowledge, this paper represents one of the initial attempts to tackle time-varying delayed SDSs with non-Markovian switching parameters, which is rigorous within this framework, yet hasn't appeared in existing literature.
  The switching process in question is highly general and encompasses independent processes, Markovian jump processes, and hidden Markov processes as special cases. 2) A unique model of time-varying delayed stochastic switching systems is proposed. Here, the switching signal is depicted by a blend of a Cox's process and a discrete-time adaptive process. 3) We derive sufficient conditions to ensure the stability of the system under study using the Lyapunov functional and the Halanay inequality approaches.

\section{Problem Formulation and Preliminaries}\label{sec2}
Throughout this paper, we use the following notations.
Take $\mathbb{R}_+=[0,+\infty)$, and $\mathbb{Z}_+$ the set of positive integers. Let $(
\Omega,
\mathcal{F},
\mathbb{P}_1)$ and $(O,Z,\mathbb{P}_2)$ be complete probability spaces, and $(\widetilde{\Omega},\mathcal{H},\mathbb{P})$ the product probability space (i.e., $\widetilde{\Omega}=O\times\Omega$, $\mathcal{H}=Z\otimes\mathcal{F}$, and $\mathbb{P}=\mathbb{P}_2\otimes\mathbb{P}_1$). 

The switching signal $r(t)$ is defined in terms of two stochastic processes defined on $(
\Omega,
\mathcal{F},
\mathbb{P}_1)$, i.e., a discrete adapted process $\{\xi^{k}\}_{k\in\mathbb{Z}_+}$ with respect to the filtration $\{\F_k:k\in\mathbb{Z}_+\}$ which satisfies the usual conditions and a Cox process $\{t_{k}\}_{k\in\mathbb{Z}_+}$ with state space $\mathbb{R}_+$ (see \cite{Cox,new_cox} for more details).
Specifically, denote by $N_{[t_{1},t_{2})}$, $N_{[t_{1},t_{2}]}$, $N_{(t_{1},t_{2})}$ and $N_{(t_{1},t_{2}]}$ the counting numbers of the switches in $[t_{1},t_{2})$, $[t_{1},t_{2}]$, $(t_{1},t_{2})$ and $(t_{1},t_{2}]$, respectively.
Suppose that each event occurs independently of others. Setting $N(t)=N_{[0,t)}$ with $N(0)=1$, the event rate $\lambda(t)$ depending on the adapted process $\{\xi^{k}\}_{k\in\mathbb{Z}_+}$ is defined by $\lambda(t)=\mu(\xi^{N(t)})$ where $\mu(\cdot)$ is a positive measurable function with respect to $(\Omega,\F_{N(t)})$ satisfying $\mu(\xi)\leq\mu_{0}$ for some positive $\mu_0$ and all $\xi\in\Omega$.
This means that, given $\{\xi^{k}\}_{k\in\mathbb{Z}_+}$,
$\mathbb{P}_1(N(t,t+h]\geq 1|~\{\xi^{k}\}_{k\in\mathbb{Z}_+})=\mu(\xi^{N(t)})h+o(h)$,
$\mathbb{P}_1(N(t,t+h]\geq 2|~\{\xi^{k}\}_{k\in\mathbb{Z}_+})=o(h)$, as $h\rightarrow0$.
Note that the point process $N(t)$ is adapted to some filtration $\{\J_{t}:t\in\mathbb{R}_+\}$.
Thus, a right-continuous $r(t)$ can be explicitly defined as $r(t)=\xi^k$, $t\in[t_{k-1},t_k)$, $k\in\mathbb{Z}_+$. 

Let $\{W_t\}_{t\in\mathbb{R}_+}$ be a standard $m$-dimensional Brownian motion, which is independent of $r(t)$, defined on $(O,Z,\mathbb{P}_2)$ with natural filtration $\{Z_t:t\in\mathbb{R}_+\}$.
And thus $(W_{t},r(t))$ defined on $(\widetilde{\Omega},\mathcal{H},\mathbb{P})$ is a well-defined adapted process
 with respect to a joint filtration $\{\HS_t:t\in\mathbb{R}_+\}$, where $\HS_{t}=Z_{t}\otimes\J_{t}\otimes\F_{N_{t}}$ with $\HS_{0}$ including all null sets.

Let $\tau(t)$ be a positive real-valued function defined on $\mathbb{R}_+$ such that
1) ${\tau}_{*}:=\inf\{\tau(t):t\geq 0\}>0$;
2) $\tau_b:=\sup\{\tau(t)-t:t\geq 0\}>0$.
For simplicity of presentation, we suppose $\tau(t)\in C^1(\mathbb{R}_{+})$, which is possibly unbounded.
Let $C([-\tau_b, 0];\mathbb{R}^n)$ denote the family of continuous function $\phi$ from $[-\tau_b, 0]$ to $\mathbb{R}^n$ with the norm
$\Vert \phi\Vert=\sup_{-\tau_b\leq\theta\leq 0}|\phi(\theta)|$, where $|\cdot|$ is the Euclidean norm in $\mathbb{R}^n$.
Let $C^b_{\mathcal{H}_0}([-\tau_b,0];\mathbb{R}^n)$ denote the family of all bounded, $\mathcal{H}_0$-measurable,
$C([-\tau_b, 0];\mathbb{R}^n)$-valued random variables.

In this paper, we consider the following TVDSDS-SSPs:
\begin{align}
\mathrm{d}x(t)=&f(x(t),x(t-\tau(t)),r(t),t)\mathrm{d}t+g(x(t),x(t-\tau(t)),r(t),t)\mathrm{d}W_{t},\label{ds}
\end{align}
with the initial condition $x(\theta)=\phi(\theta)$, $\theta\in[-\tau_b,0]$ and $\phi\in C_{\HS_{0}}^{b}([-\tau_b,0])$ where $x(t)\in\mathbb R^{n}$ is the state vector, $W_{t}$, $r(t)$ and $\tau(t)$ are the Brownian motion, the switching signal and time-varying delay defined above. $f:\mathbb R^{n}\times\mathbb R^{n}\times\Omega\times\mathbb{R}_+\to\mathbb R^{n}$ and $g: \mathbb R^{n}\times\mathbb R^{n}\times\Omega\times\mathbb{R}_+\to \mathbb R^{n\times m}$ are nonlinear functions, for which we impose the following hypotheses.

$\mathbf H_{1}$: $f(x,y,r,t)$ and $g(x,y,r,t)$ satisfy {\em uniformly local Lipschitz} and {\em linearly growing} condition with respect to $x$ and $y$. That is, there exist positive constant sequence $h_{k}>0$ ($k=1,2,\cdots$) and $h>0$ independent of $\xi$ and $t$ such that
1) $|f(x,y,\xi,t)-f(x',y',\xi,t)|+|g(x,y,\xi,t)-g(x',y',\xi,t)|\le h_{k}(|x-x'|+|y-y'|)$ holds for any $\xi\in\Omega$, $t\ge 0$ and $x,y,x',y'\in\mathbb R^{n}$ with $\max\{|x|,|y|,|x'|,|y|'\}\le k$; 2) $|f(x,y,\xi,t)|+|g(x,y,\xi,t)|\le h(1+|x|+|y|)$ holds for all $\xi\in\Omega$, $t\ge 0$ and $x,y\in\mathbb R^{n}$.

To better characterize the stability of TVDSDS-SSPs, we present the following definitions.

\begin{definition}\label{def:mss}
The time-varying delayed stochastic dynamic system (\ref{ds}) is said to be \emph{mean-square stable} if $\mathbb{E}(|x(t)|^2)\rightarrow 0$, as $t\rightarrow +\infty$ for arbitrary initial condition $\phi\in C_{\mathcal{H}_0}^b([-\tau_b,0])$.
\end{definition}
\begin{remark}
Definition \ref{def:mss} is consistent with the system with Markovian jumping parameters in \cite{Costa}.
\end{remark}
\begin{definition}\label{def:nu-mss}
   Given positive $\nu(t)\in C^1([-\tau_b,+\infty))$ satisfying
   $\nu(t)\rightarrow+\infty$, as $t\rightarrow +\infty$,
   the time-varying delayed stochastic dynamic system (\ref{ds}) is said to be
\emph{$\nu$-mean-square stable} if for arbitrary initial condition $\phi\in C_{\mathcal{H}_0}^b([-\tau_b,0])$,
there exists constant $M>0$ such that $\mathbb{E}(|x(t)|^2)\leq {M}/{\nu(t)}$, as $t\geq 0$.
\end{definition}

\begin{definition}\label{def:as}
The time-varying delayed stochastic dynamic system (\ref{ds}) is said to be
\emph{stable in probability} if $\mathbb{P}(|x(t)|>\epsilon)\rightarrow 0$, as $t\rightarrow +\infty$ for arbitrary initial condition $\phi\in C_{\mathcal{H}_0}^b([-\tau_b,0])$ and any $\epsilon>0$.
\end{definition}

\begin{definition}\label{def:nu-as}
   Given positive $\nu(t)\in C^1([-\tau_b,+\infty))$ satisfying
   $\nu(t)\rightarrow+\infty$, as $t\rightarrow +\infty$,
   the time-varying delayed stochastic dynamic system (\ref{ds}) is said to be
\emph{$\nu$-stable in probability} if for arbitrary initial condition $\phi\in C_{\mathcal{H}_0}^b([-\tau_b,0])$,
there exists a positive random variable $\xi$ satisfying $\mathbb{P}(\xi>R)\rightarrow0$ as $R\rightarrow +\infty$,
such that $\mathbb{P}\left(|x(t)|^2\leq {\xi}/{\nu(t)}\right)\rightarrow1$ as $t\rightarrow +\infty$.
\end{definition}
\begin{remark}\label{lemma:nu-stability}
Definitions \ref{def:nu-mss} and \ref{def:nu-as} encompass a variety of convergence concepts based on the selection of $\nu(\cdot)$. Options such as $\nu(t)=\exp(\alpha t)$, $\nu(t)=(t+\tau_b+1)^{\alpha}$ for some $\alpha>0$, and $\nu(t)=\ln(t+\tau_b+1)$ (or $\nu(t)=\ln\ln(t+\tau_b+3)$) result in exponential stability, power stability, Log stability (or Log-Log stability), respectively \cite{ChenT}.
\end{remark}

\section{Solution of Cauchy Problem and Generalized Dynkin's formula}

Under Hypothesis $\mathbf H_{1}$, the existence and uniqueness of the solution to (\ref{ds}) are established by the following theorem.

\begin{theorem}\label{ext}
Assuming that Hypothesis $\mathbf H_{1}$ is valid and $\tau(t)$ and $r(t)$ are defined as in Section \ref{sec2}, with the event rate function $\mu(\cdot)$ measurable and satisfying $\sup_{\xi\in\Omega}\mu(\xi)\le\mu_{0}$ for a positive $\mu_{0}$, equation (\ref{ds}) with the initial function $\phi\in C_{\HS_{0}}^{b}([-\tau_b,0])$ possesses a unique continuous $\HS_{t}$-adapted solution where $\E(\sup_{-\tau_b\le s\le t}|x(s)|^{p})<+\infty$ holds for any $t$ and any $p\ge 1$.
\end{theorem}

\begin{proof}
Note that the event rate of the concerned switching process satisfies $\lambda(t)\le\mu_{0}$ for all $t$ with probability one. Then, in any finite time $[0,t]$, only a finite number of switches occur at $t_{1},\cdots,t_{N}$.
We first consider the time-varying delayed stochastic dynamical system (\ref{ds}) on $[0,t_1]$ as follows:
\begin{align*}
   \mathrm{d} x(t)=&f(x(t),x(t-\tau(t)),r(0),t)\mathrm{d}t+g(x(t),x(t-\tau(t)),r(0),t)\mathrm{d}W_{t},\quad t\in[0,t_1],
\end{align*}
with initial data $\phi\in C_{\mathcal{H}_0}^b([-\tau_b,0];\mathbb{R}^n)$.
Since $\tau(t)\geq\tau_b^*$,
\begin{align*}
   \mathrm{d} x(t)=&f(x(t),\phi(t-\tau(t)),r(0),t)\mathrm{d}t+g(x(t),\phi(t-\tau(t)),r(0),t)\mathrm{d}W_{t},\quad t\in[0,\tau_b^*\land t_1],
\end{align*}
where $\land$ gives the smaller between two numbers. By the existence and uniqueness theorem of stochastic differential delay equations \cite{Mao2},
under Hypothesis $\mathbf H_{1}$, there exists a unique solution.
Repeating such a process on $[\tau_b^*\land t_1,2\tau_b^*\land t_1]$, $[2\tau_b^*\land t_1,3\tau_b^*\land t_1]$, etc., we can see that the solution is unique on $[-\tau_b,t_1]$, and
$\mathbb{E}(\sup_{-\tau_b\leq t\leq t_1}|x(t)|^2)<+\infty$.
Following the same proof line in \cite{Mao1}, for almost every sequence $r(t)$, for any $t$ and any $p\geq 1$, a unique continuous solution exists in $[0, t]$, and
$\mathbb{E}(\sup_{-\tau_b\leq s\leq t}|x(s)|^p)<+\infty$.
\end{proof}

In what follows, we study the infinitesimal generator of the Lyapunov functional. For a real-valued bounded $\HS_{t}$-adapted process $f_{t}$ and $s\ge 0$, define a semigroup of conditioned shifts as $\T(s)f_{t}=\E(f(t+s)|\HS_{t})$.
Then, the semigroup is generated by
$\T(u)\T(s)f_t=\E(f(t+s+u)|\HS_{t})$.
It can be seen that $\T(s)$ is a contraction semigroup with the norm $
\|f\|=\sup_{t\ge 0}\E(|f(t)|)$,
which comprises a process norm space $\mathcal L$. Let $\mathcal L_{0}$ be the subspace of $\mathcal L$ in which $\T(s)$ is $p$-right continuous. Then, in $\mathcal L_{0}$, a {\em point-wise} infinitesimal generator is defined as
\begin{align*}
\A f=p-\lim_{h\to 0+}\frac{1}{h}(\T(h)f-f),
\end{align*}
where $p-\lim_{h\to 0+}$ means point-wise limit. If this limit is finite in the subspace $D(\A)\subset\mathcal L_{0}$, $\A$ is called to be well-defined, and $D(\A)$ is the definition region of $\A$. Similar to what has been done in \cite{Dynkin}, we have the following result.
\begin{lemma}(Generalized Dynkin's formula \cite{Dynkin,MT}) \label{thm3} For each right-continuous $f\in D(\A)$ and a stopping time $\tau\ge t$ that almost surely satisfy $\sup_{t\le\tau}|f(t)-f(0)-\int_{0}^{t}\A f(s)\mathrm{d}s|<+\infty$, we have
\begin{align}
\E(f(\tau)|\HS_{t})-f(t)=\E(\int_{t}^{\tau}\A f(s)\mathrm{d}s|\HS_{t}).\label{DK}
\end{align}
\end{lemma}

Let $x_{t}(\theta)=x(t+\theta)$ with $\theta\in[-\tau_b,0]$ and select a candidate Lyapunov functional as follows:
\begin{align}
V(x_{t},t,r(t))=P(x(t),t,r(t))+\int_{t-\tau(t)}^{t}Q(x(\theta),\theta)\mathrm{d}\theta,\label{V}
\end{align}
which belongs to $\mathcal D(\A)$. Then, the formula (\ref{DK}) still holds. In order to derive the infinitesimal generator of $V(x_{t},t,r(t))$ given by (\ref{V}), the following lemma is needed.
\begin{lemma}\label{lem1}
Consider the switching signal process $r(t)$ described in Section \ref{sec2} and let $k_{*}=N(t)$ and $r(t)=\xi^{k_{*}}$.
For an infinitesimal time interval $h$ and any bounded measurable function $q(\cdot):\Omega\to\mathbb R^{d}$ where the integer $d>0$, we have
\begin{align}
&\E\left[q(r(t+h))|\mathcal{H}_{t}\right]-q(r(t))
=h\mu(\xi^{k_{*}})\left\{\E\left(q(\xi^{k_{*}+1})|\F_{k_{*}}\right)-q(\xi^{k_{*}})\right\}+o(h),\label{cox}
\end{align}
as $h\rightarrow 0+$.
\end{lemma}
\begin{proof}For the  given sequence of the discrete-time process $\{\xi^{k}\}_{k\in\mathbb{Z}_+}$, in the time interval $(t,t+h]$,
$\Prob_1(r(t+h)=\xi^{k_{*}+1}|\{\xi^{k}\}_{k\in\mathbb{Z}_+})=h\mu(\xi^{k_{*}})+o(h)$
 as $h\rightarrow 0+$.  Thus, we have
$\E[q(r(t+h))|~\{\xi^{k}\}_{k\in\mathbb{Z}_+}]
=q(\xi^{k_{*}})
(1-h\mu(\xi^{k_{*}}))+q(\xi^{k_{*}+1})h\mu(\xi^{k_{*}})+o(h)$.
Then, taking conditional expectation on both sides of the above equation, we have
$\E[q(r(t+h))|\mathcal{H}_{t}]
=q(\xi^{k_{*}})(1-h\mu(\xi^{k_{*}}))+\E[q(\xi^{k_{*}+1})|\F_{k_{*}}]h\mu(\xi^{k_{*}})+o(h)$, which completes the proof.
\end{proof}

By letting $r(t)=\xi^{k}$, the infinitesimal generator of $V(x_{t},t,r(t))$ is provided in the following lemma.
\begin{lemma}\label{le3}
If $P(x, t, \xi)$ is continuous with respect to $(x, t, \xi)$, $C^1$ with respect to $(x,t)$,
$C^2$ with respect to $x$,
and there exists positive constant $C_{x,t}$ such that $P(x,t,\xi)\leq C_{x,t}$ for all $\xi\in\Omega$, and $Q(x,t)$ is measurable with respect to $(x,t)$,
then the (point-wise) infinitesimal generator of $V(x_{t},t,r(t))$ is given by
\begin{align}\label{AV}
   \hat{\mathcal{A}}V &=\frac{\partial P}{\partial t} +\frac{\partial P}{\partial x}f(x(t), x(t-\tau(t)), \xi^k, t)\\
   &+\frac{1}{2}\text{tr}\Big[g^{\top}(x(t), x(t-\tau(t)), \xi^k, t)
   \frac{\partial^2 P}{\partial x^2}
   g(x(t), x(t-\tau(t)), \xi^k, t)\Big]\nonumber\\
   &+\mu(\xi^k)\left\{\mathbb{E}[P(x(t),t,\xi^{k+1})|\mathcal{F}_k]-P(x(t),t,\xi^k)\right\}
   \nonumber\\
   &+Q(x(t),t)
   -(1-\tau^{\prime}(t))Q(x(t-\tau(t)),t-\tau(t)).\nonumber
\end{align}
\end{lemma}
\begin{proof}Consider the difference quotient, by ignoring the higher order terms and letting $\delta x=x(t+h)-x(t)$, we have
\begin{align*}
&{[V(x_{t+h},t+h,r(t+h))-V(x_{t},t,r(t))]}/{h}\\
=&\frac{\partial P}{\partial x}(x(t),t+h,r(t+h))\frac{\delta x}{h}
+\frac{1}{2h}\delta x^{\top}\frac{\partial^{2}P}{\partial x^{2}}(x(t),t+h,r(t+h))\delta x\\
+&\frac{\partial P}{\partial t}(x(t),t,r(t+h))
+\frac{1}{h}[P(x(t),t,r(t+h))-P(x(t),t,r(t))]+o(h).
\end{align*}
Then, by taking conditional expectation on both sides of the above equation with respect to $\HS_{t}$,
according to Lemma \ref{lem1}, the right continuity of $r(t)$, continuity of $P(\cdot,\cdot,\cdot)$ and independence of the components of $W_{t}$, noting $\E(\mathrm{d}W_{t+s}|\HS_{t})=0$ for all $s>0$, and taking limit as $h\to 0+$, it is obtained that
\begin{align*}
   &\hat{\mathcal{A}}P(x(t),t,r(t)) \\
   =& \frac{\partial P}{\partial t} +\frac{\partial P}{\partial x}f(x(t), x(t-\tau(t)), \xi^k, t)\\
   +&\frac{1}{2}\text{tr}\Big[g^{\top}(x(t), x(t-\tau(t)), \xi^k, t)\frac{\partial^2 P}{\partial x^2} g(x(t), x(t-\tau(t)), \xi^k, t)\Big]
   \\
   +&\mu(\xi^k)\left\{\mathbb{E}[P(x(t),t,\xi^{k+1})|\mathcal{F}_k]-P(x(t),t,\xi^k)\right\}.
\end{align*}
Moreover, it is easy to verify that $\A \int_{t-\tau(t)}^{t}Q(x(\theta),\theta)\mathrm{d}\theta=Q(x(t),t)-(1-\tau^{\prime}(t))Q(x(t-\tau(t)),t-\tau(t))$, which completes the proof.
\end{proof}

The infinitesimal generator of $V(x_{t},t,r(t))$ given in Lemma \ref{le3} has the term
\begin{align}
\chi_{k}=\mu(\xi^{k})\left\{\E\left[P(x(t),t,\xi^{k+1})|\F_{k}\right]-P(x(t),t,\xi^{k})\right\},
\label{chi}
\end{align}
computed in the following cases of different types of $r(t)$.

\paragraph{Independent switching} Consider an independent process $\{\xi^k\}$, possibly non-identically distributed.
Letting $r(t)=\xi^k$ with $k=N(t)$, since
$\E[\xi^{k+1}|\F_{k}]=\E[\xi^{k+1}]$,
we have
$\chi_k=\mu(\xi^k)\{
    \mathbb{E}[P(x(t),t,\xi^{k+1})]-P(x(t),t,\xi^{k})
\}$.

\paragraph{Homogeneous Markovian switching} Consider a Markovian jumping process induced by a homogeneous continuous Markov chain with state space $\Omega$ and transition probability $\TT(\mathrm{d}\xi',\xi)$. Letting $r(t)=\xi^{k}$, $k=N(t)$, and using the Markov property, (\ref{chi}) becomes 
$\chi_{k}=\mu(\xi^{k})\left\{\int_{\Omega}P(x(t),t,\xi')\TT(\mathrm{d}\xi',\xi^{k})-P(x(t),t,\xi^{k})\right\}$.

In particular, if $\Omega$ has finite states, i.e., $\Omega=\{\xi_{i}\}_{i=1}^{K}$ with a transition probability matrix $R=[r_{ij}]_{i,j=1}^{K}$ with $r(t)=\xi_{i}$, (\ref{chi}) becomes $\chi_k=\mu(\xi_i)\{\sum_{j=1}^Kr_{ij}P(x(t), t, \xi_j)-P(x(t),t,\xi_i)\}$.

\paragraph{Hidden Markov process} Let $\{\zeta^{k},\xi^{k}\}$ be a hidden Markov process, where $\{\zeta^{k}\}$ is a homogeneous Markov chain with transition probability $\TT(\mathrm{d}\zeta',\zeta)$ on state space $\Omega_{1}$, and $\xi^{k}$ is dependent on $\zeta^{k}$ with a homogeneous conditional probability $\Prob(\mathrm{d}\xi^{k}|\zeta^{k})$ for all $k$. Then (\ref{chi}) becomes
$\chi_{k}=\mu(\xi)\times\{\int_{\Omega}P(x(t),t,\xi)\int_{\Omega_{1}}
\Prob(\mathrm{d}\xi|\zeta')\TT(\mathrm{d}\zeta',\zeta^{k})
-P(x(t),t,\xi^{k})\}$.

\section{Stability Analysis}\label{sa}

In this section, the stability issue is investigated for the time-varying delayed stochastic dynamical system (\ref{ds}). Two stability criteria are derived by using two different analysis approaches. Herein, we present the following Doob's first convergence theorem which plays an essential role for our main results.

\begin{lemma}\label{lem3}(Doob's first convergence theorem \cite{Oks})
   Let $N_{t}$ be a right continuous supermartingale with the property $\sup_{t>0}\E[N_{t}^{-}]<+\infty$, where $N_{t}^{-}=\max\{-N_{t},0\}$.
Then, the pointwise limit $N(\omega) =\lim_{t\rightarrow+\infty}N_t(\omega)$ exists for almost all $\omega$.
\end{lemma}

Firstly, the stability analysis is performed by the approach based on the Lyapunov functional with extension to the stochastic case. Let $V(x_{t},t,\xi)$ be a (random) functional and the corresponding sufficient condition is provided under which the solution $x(t)$ to the time-varying delayed stochastic dynamical system (\ref{ds}) converges to zero almost surely.

\begin{theorem}\label{LST}
Suppose that Hypothesis $\mathbf H_{1}$ holds and the following conditions are satisfied:
\begin{enumerate}
\item {
   $V(x_t,t,\xi)$ is continuous with respect to ($x_t, t, \xi$),
       $V(x_t,t,\xi)\geq\nu(t)\omega(x(t))$ for all $t\geq 0$, all $\xi\in\Omega$,
       and some non-negative function
       $\omega(\cdot)$, non-random $\nu(t)> 0$ with $\lim_{t\rightarrow+\infty}\nu(t)=+\infty$;
   }
\item $\A V(x_{t},t,r(t))\le 0$ for all $t\ge 0$ and $r(t)\in\Omega$;
\item $\E V(\phi,0,r(0))<+\infty$ for any $\phi\in C^b_{\mathcal{H}_0}([-\tau_b,0])$.
\end{enumerate}
Then, we have the following results:
\begin{enumerate}
\item{
   $\lim_{t\to+\infty}\E[\omega(x(t))]=0$;
}
\item{
  the system is stable in probability, if $\omega(x)$ is continuous and strictly positive definite, i.e., $\omega(x)=0$ iff $|x|=0$.
}
\item{
  the system is $\nu$-mean-stable and $\nu$-stable in probability, if there exists positive constants $c_1\leq c_2$ such that $c_1|x|^2\leq \omega(x)\leq c_2|x|^2$.
}
\end{enumerate}
\end{theorem}
\begin{proof}For the purpose of notation simplicity, $V(x_{t},t,r(t))$ is denoted by $V(t)$. From the generalized Dynkin's formula (\ref{DK}), one has
\begin{align*}
\E(V(t)|\HS_{s})=V(s)+\E\left(\int_{s}^{t}\A V(\theta)\mathrm{d}\theta|\HS_{s}\right)\le V(s)
\end{align*}
for any $s\le t$, which implies that $V(t)$ is supermartingale.
Taking expectation on both sides of the above inequality, we have
$\nu(t)\E(\omega(x(t)))\le\E(V(t))\le\E(V(0))$.
Thus, conclusion 1) follows by noting $\lim_{t\to+\infty}\nu(t)=+\infty$.

In addition, since $V(t)$ is a non-negative continuous supermartingale, from Doob's first martingale convergence theorem (Lemma \ref{lem3}), it is known that there exists a random variable $V_{\infty}$ such that $\lim_{t\to+\infty}V(t)=V_{\infty}$
almost surely.
Thus $\lim_{t\to+\infty}\mathbb{P}(|V(t)-V_{\infty}|\geq 1)=0$.
Moreover, since $V(0)\in L^1$, we have $V_{\infty}\in L^1$, which implies that
$\mathbb{P}(V_{\infty}=+\infty) = 0$.
Note that
$\mathbb{P}\left(w(x(t))<{(V_{\infty}+1)}/{\nu(t)}\right)\geq
\mathbb{P}(V(t)<V_{\infty}+1)$,
thus
$\mathbb{P}\left(
      w(x(t))<{(V_{\infty}+1)}/{\nu(t)}
   \right)\rightarrow 1$, as $t\rightarrow+\infty$.
When $\omega(x)$ is continuous and strictly positive definite, conclusion
2) holds directly by Definition \ref{def:as}.
When $c_1|x|^2\leq \omega(x)\leq c_2|x|^2$,
we can derive $\nu$-mean-stability directly by Definition \ref{def:nu-mss},
and $\nu$-stability in probability by letting $\xi=V_{\infty}+1$ in Definition \ref{def:nu-as}.
\end{proof}

\begin{remark}
If $r(t)$ is a homogeneous Markov chain with finite states, i.e., the well-known Markovian jumping process, then the results in Theorem \ref{LST} can be analog to those in \cite{Mao1}.
\end{remark}

Secondly, we develop an approach (similar to the celebrated Halanay inequality \cite{LLC}) to obtain the stability criterion for the time-varying delayed stochastic dynamical system (\ref{ds}). For simplicity of presentation, we denote  
$V(x_\theta,\theta,r(\theta)):=V(\phi,0,r(0))$, $\theta\in[-\tau_b,0)$, and give the following conclusion.
\begin{theorem}\label{Hal}
Suppose that Hypothesis $\mathbf H_{1}$ holds and the following conditions are satisfied:
\begin{enumerate}
\item {
$V(x_t,t,\xi)$ is continuous with respect to ($x_t, t, \xi$),
    $V(x_t,t,\xi)\geq\nu(t)\omega(x(t))$ for all $t\geq 0$, all $\xi\in\Omega$,
    and some non-negative function
    $\omega(\cdot)$, non-random $\nu(t)> 0$ with $\lim_{t\rightarrow+\infty}\nu(t)=+\infty$;
}
\item
{
   $\hat{\mathcal{A}}V(x_t,t,r(t))\leq J(t) -\alpha(t)V(x_t,t,r(t))+\beta(t)\psi(x_t,t,r(t))$ for some non-random
right continuous positive $\alpha(t)$ and $\beta(t)$, some non-negative continuous functional $\psi(\cdot)$, and
some positive $\mathcal{H}_t$-adapted right-continuous process $J(t)$ with $\mathbb{E}(J(t))\leq J_0$ for all $t\geq 0$
and some positive constant $J_0$;
}

\item $\alpha(t)-\beta(t)\ge \eta$ for all $t\ge 0$ and some $\eta>0$;

\item for all continuous function $y$ and time $s\geq 0$ and $t\geq s$,
\begin{align*}
   \E[\psi(y_{t},t,r(t))|\HS_{s}]\le\sup_{t-\tau(t)\le \theta\le t}\E[V(y_\theta,\theta,r(\theta))|\HS_{s}];
   \end{align*}
\item {
   $\E(V(\phi,0,r(0)))<+\infty$ for any $\phi\in C^b_{\mathcal{H}_0}([-\tau_b,0])$.
}
\end{enumerate}
Then, $\lim_{t\to+\infty}\E[\omega(x(t))]=0$. Furthermore, by replacing $J(t)\leq J_1$ for all $t$ and some positive constant $J_1$ in condition 2), then we have the following results:
\begin{enumerate}
   \item{
      the system is stable in probability, if $\omega(x)$ is continuous and strictly positive definite, i.e., $\omega(x)=0$ iff $|x|=0$.
    }
    \item{
      the system is $\nu$-mean-stable and $\nu$-stable in probability, if there exists positive constants $c_1\leq c_2$ such that $c_1|x|^2\leq \omega(x)\leq c_2|x|^2$.
    }
\end{enumerate}
\end{theorem}

\begin{proof} Denoting $V(x_t,t,r(t))=V(t)$ for simplicity, from the generalized Dynkin's formula (\ref{DK}), we have
$\E(V(t)|\HS_{s})=
\E\left[\int_{s}^{t}\A V(t')\mathrm{d}t'|\HS_{s}\right]+V(s)$.
Then by conditions 2) and 3), we have
\begin{align*}
&\E(V(t)|\HS_{s})
\le V(s)+\int_{s}^{t}\Bigg[-\alpha(t')\E(V(t')|\HS_{s})+\beta(t')\sup_{t'-\tau(t')\le\theta\le t'}\E(V(\theta)|\HS_{s})+\E(J(t')|\HS_{s})\Bigg]\mathrm{d}t'.
\end{align*}
Hence, with a fixed $s\geq 0$ and for any $t\ge s$, the upper-right Dini derivative of $\E(V(t)|\HS_{s})$ exists and is given by
\begin{align}
D^{+}\E(V(t)|\HS_{s})\le& -\alpha(t)\E(V(t)|\HS_{s})+\E(J(t')|\HS_{s})+\beta(t)\sup_{t-\tau(t)\le\theta\le t}\E(V(\theta)|\HS_{s}).\label{Dini}
\end{align}
Taking expectation on both sides of (\ref{Dini}), we further have
\begin{align*}
D^{+}\E(V(t))\le -\alpha(t)\E(V(t))+\beta(t)\sup_{t-\tau(t)\le\theta\le t}\E(V(\theta))+J_{0}.
\end{align*}
By employing the generalized Halanay inequality \cite{LLC}, it can be obtained that $
\E(V(t))\le\max\{(J_{0}/\eta),\E(V(0))\}$
for all $t\ge 0$. This implies 
\begin{align*}
	\E(\omega(x(t)))\le(1/\nu(t))\max\{(J_{0}/\eta),\E(V(0))\}.
\end{align*}
Therefore, we have $\lim_{t\to+\infty}\E(\omega(x(t)))=0$.

Under the additional conditions, it follows from (\ref{Dini}) that
\begin{align*}
D^{+}\E(V(t)|\HS_{s})\le& -\alpha(t)\E(V(t)|\HS_{s})+J_{1}+\beta(t)\sup_{t-\tau(t)\le\theta\le t}\E(V(\theta)|\HS_{s}).
\end{align*}
Then, by using the generalized Halanay inequality \cite{LLC} again, we obtain
$\E(V(t)|\HS_{s})\le\max\left\{{J_{1}}/{\eta},V(s)\right\}$. Define $W(t)=\max\{V(t),J_{1}/\eta\}$, and we have $\E(W(t)|\HS_{s})\le W(s)$, which means that $W(t)$ is a continuous non-negative supermartingale.
By Lemma \ref{lem3}, it is known that there exists a random variable $W_{\infty}$ such that $\lim_{t\to+\infty}W(t)=W_{\infty}$ holds
almost surely.
Then, conclusions 1) and 2) hold by following the same proof line of Theorem \ref{LST}.
\end{proof}

\begin{remark}
It should be pointed out that our results allow that each subsystem is not necessarily stable according to the presented Lyapunov functional/Halanay inequality.
The numerical examples in Section \ref{sec:examples} will show this argument.
\end{remark}

\section{Stability of Switching Stochastic Nonlinear Systems}

In this section, the results derived in Section \ref{sa} are applied to a special class of time-varying delayed SDSs, and the corresponding stability criteria are established.

Consider a switching delayed networked system with stochastic perturbations as follows:
\begin{align}\label{nn}
   \mathrm{d}x(t)=&[-D(r(t))x(t)+A(r(t))g(x(t))+B(r(t))g(x(t-\tau(t)))]\mathrm{d}t
+\sigma(g(x(t)),g(x(t-\tau(t))),r(t))\mathrm{d}W_{t},
 \end{align}
where $x(t)=[x_{1}(t),\cdots,x_{n}(t)]^{\top}$ stands for the state vector with $x_{i}(t)$ being the state of node $i$, $i=1,\cdots,n$; $D(r(t))=\text{diag}[D_{1}(r(t)),\cdots,D_{n}(r(t))]\in\mathbb R^{n\times n}$ is a positive definite diagonal matrix; $A(r(t))$ and $B(r(t))\in\mathbb R^{n\times n}$ are the connection weight matrix and the delayed connection weight matrix, respectively; $g(x)=[g_{1}(x_{i}),\cdots,g_{n}(x_{n})]^{\top}$ denotes the nonlinear output function; $\sigma(\cdot,\cdot,\cdot)$ is the noise intensity function; and $r(t)$, $W_{t}$ and $\tau(t)$ are defined as the same in Section \ref{sec2}.

\begin{remark}
We point out that system (\ref{nn}) can be regarded as a stochastic extension of the delayed neural network\cite{new_WFZ,new_WSSC}, incorporating the influence of both non-Markovian switching and time-varying delays.
\end{remark}

For $g$ and $\sigma$, we impose the following hypotheses.

$\mathbf H_{2}$: The nonlinear function $g(\cdot)$ is $C^{1}$ and satisfies
\begin{enumerate}
\item  $g_{i}(0)=0$;
\item $0\le g_i'(\cdot)\le G_{i}$ for all $i=1,\cdots, n$.
\end{enumerate}

$\mathbf H_{3}$: The noise intensity function $\sigma(\cdot,\cdot,\cdot)$ satisfies
\begin{enumerate}
\item uniformly local Lipschitz and linear growth condition;
\item {
   there exist measurable semi-positive definite matrix-valued function $E(\cdot)$ and $F(\cdot)$ such that for all $x,y\in\mathbb{R}^n$ and all $\xi\in\Omega$, it holds 
   \begin{align*}
     \text{tr}\left[\sigma^{\top}(x,y,\xi) \sigma(x,y,\xi)\right]\leq x^{\top}E(\xi)x +y^{\top}F(\xi) y;
   \end{align*}
}
\item {
   there exist measurable positive definite matrix-valued function $P(\cdot)$,
    positive continuous function $a(\cdot)$ and positive constants $C_1\leq C_2$ such that $
   C_1 I_n\leq P(\xi) \leq C_2 I_n$ holds for all $\xi\in\Omega$, and
    \begin{align*}
        \text{tr}\left[\sigma^{\top}(x,y,\xi) P(\xi)\sigma(x,y,\xi)
        \right]&\leq a(\xi)x^{\top}P(\xi)x+a(\xi)y^{\top}P(\xi)y,
    \end{align*}
holds for all $x,y\in\mathbb{R}^n$ and all $\xi\in\Omega$.
}
\end{enumerate}

The following result is a direct consequence of Theorem \ref{LST}.
\begin{theorem}\label{thmnn1}
Suppose that Hypotheses $\mathbf H_{2}$ and $\mathbf H_{3}$ hold and $\sup_{\xi\in\Omega}\mu(\xi)\le\mu_{0}$ for some positive $\mu_{0}$. If there exist measurable positive definite matrix-valued function $R(\cdot)$, constant diagonal positive definite matrix $Z=\text{diag}[Z_{1},\cdots,Z_{n}]$, constant positive definite matrix $Q$,
non-random positive $\nu(t)\in C^1([-\tau_b,+\infty))$ and positive constants $\alpha_\nu,\beta_\nu$ satisfying
\begin{align*}
   \lim_{t\rightarrow+\infty}\nu(t)=+\infty,
        \quad \sup_{t\geq0}\frac{\nu^{\prime}(t)}{\nu(t)}\leq\alpha_{\nu},
   \quad \inf_{t\geq0}(1-\tau^{\prime}(t))\frac{\nu(t-\tau(t))}{\nu(t)}\geq\beta_{\nu},
\end{align*}
such that
\begin{align}
\Pi^k=\left[\begin{array}{lll}\Sigma^{k}&P(\xi^{k})A(\xi^{k})&P(\xi^{k})B(\xi^{k})\\
*&\Lambda^{k}&ZB(\xi^{k})\\
*&*&\Gamma^{k}\end{array}\right]\leq 0\label{pi}
\end{align}
holds, where
\begin{align*}
\Sigma^{k}=&\alpha_{\nu}\left(P(\xi^k)+ZG\right)-2P(\xi^{k})D(\xi^{k})
+R(\xi^{k})\\
&+\mu(\xi^{k})\bigg\{\E[P(\xi^{k+1})|\F_{k}]-P(\xi^{k})\bigg\},\\
\Lambda^{k}=&-2ZD(\xi^{k})G^{-1}+ZA(\xi^{k})+A^{\top}(\xi^{k})Z-G^{-1}R(\xi^{k})G^{-1}\\
&+a(\xi^{k})P(\xi^{k})+Q+\lambda_{\max}(ZG)E(\xi^k),\\
\Gamma^{k}=&-\beta_\nu Q+a(\xi^k)P(\xi^k)+\lambda_{\max}(ZG)F(\xi^k),\\
G=&\text{diag}[G_{1},\cdots,G_{n}],
\end{align*}
then the system (\ref{nn}) is $\nu$-stable in probability.
\end{theorem}
\begin{proof}
First, the existence and uniqueness of (\ref{nn}) is ensured by Theorem \ref{ext} since system (\ref{nn}) satisfies Hypothesis $H_{1}$. Consider a candidate Lyapunov functional as follows:
\begin{align*}
   V_1&(x_t,t,r(t))
   =\nu(t)\big[x^{\top}(t)P(r(t))x(t)
   +2\sum_{i=1}^{n}Z_i \int_{0}^{x_i}g_i(\rho)\mathrm{d}\rho\big]
   +\int_{t-\tau(t)}^{t} \nu(s)g^{\top}(x(s))Qg(x(s))\mathrm{d}s.
\end{align*}
Denoting $\sigma_{t}=\sigma(g(x(t)),g(x(t-\tau(t))),r(t))$ and letting $r(t)=\xi^{k}$, the infinitesimal generator on $V_{1}(x_{t},t,r(t))$ can be given by
\begin{align*}
   \hat{\mathcal{A}}&V_1(x_t,t,r(t))
   =\nu^{\prime}(t)\big[
       x^{\top}(t)P(\xi^k)x(t)+2\sum_{i=1}^nZ_i
       \int_0^{x_i}g_i(\rho)\mathrm{d}\rho\big]
   \\
   &+2\nu(t)[x^{\top}(t)P(\xi^k)+g^{\top}(x(t))Z]
   [-D(\xi^k)x(t)+A(\xi^k)g(x(t))+B(\xi^k)g(x(t-\tau(t)))]\\
   &+\nu(t)\text{tr}\left(\sigma_t^{\top}[P(\xi^k)+ZG^{\prime}(x(t))]\sigma_t\right)
   +\nu(t)x^{\top}(t)\mu(\xi^k)
\{\mathbb{E}[P(\xi^{k+1})|\mathcal{F}_k]-P(\xi^k)\}x(t)\\
   &+\nu(t)g^{\top}(x(t))Qg(x(t))
   -(1-\tau^{\prime}(t))\nu(t-\tau(t))g^{\top}(x(t-\tau(t)))Qg(x(t-\tau(t)))\\
   &+\nu(t)x^{\top}(t)R(\xi^k)x(t) -\nu(t) x^{\top}(t)R(\xi^k)x(t).
\end{align*}

Under Hypothesis $\mathbf H_{2}$, it can be seen that (\romannumeral1) $\int_{0}^{x_{i}}g_{i}(\rho)d\rho\le\frac{1}{2}G_{i}x_{i}^{2}$, (\romannumeral2)
$g_{i}(x_{i})Z_{i}D_{i}x_{i}\ge g_{i}(x_{i})Z_{i}D_{i}G_{i}^{-1}g_{i}(x_{i})$,   (\romannumeral3) $x^{\top}(t)R(\xi^{k})x(t)\ge g^{\top}(x(t))G^{-1}R(\xi^{k})G^{-1}g(x(t))$, and in terms of Hypothesis $\mathbf H_{3}$, we have
\begin{align*}
\text{tr}&[\sigma_{t}^{\top}P(\xi^{k})\sigma_{t}]
\le a(\xi^{k})\Big[g^{\top}(x(t))P(\xi^{k})g(x(t))+g^{\top}(x(t-\tau(t)))P(\xi^{k})g(x(t-\tau(t)))\Big],\\
\text{tr}&[\sigma_{t}^{\top}ZG^{\prime}(x(t))\sigma_{t}]
\le \lambda_{\max}(ZG)\Big[g^{\top}(x(t))E(\xi^k) g(x(t))
+g^{\top}(x(t-\tau(t)))F(\xi^k)g(x(t-\tau(t)))\Big].
\end{align*}
Since $\sup_{t\geq0}\frac{\nu^{\prime}(t)}{\nu(t)}\leq\alpha_{\nu}$
and $\inf_{t\geq0}(1-\tau^{\prime}(t))\frac{\nu(t-\tau(t))}{\nu(t)}\geq\beta_{\nu}$,
\begin{align*}
\A V_{1}(x_{t},t,r(t))\le&\nu(t)\left[x^{\top}(t),g^{\top}(x(t)),g^{\top}(x(t-\tau(t)))\right]
\Pi^{k}\left[\begin{array}{l}
x(t)\\g(x(t))\\g(x(t-\tau(t)))\end{array}\right].
\end{align*}

Noting $\Pi^{k}\leq 0$, we have $\A V_{1}(x_{t},t,r(t))\le 0$. The rest of the proof follows directly from Theorem \ref{LST}.
\end{proof}

\begin{remark}\label{remark:3}
   It is straightforward to derive the $\nu$-mean-square stability, and the proof is omitted.
\end{remark}

The following result is obtained from Theorem \ref{Hal}.
\begin{theorem}\label{thmnn2} Suppose that Hypotheses $\mathbf H_{2}$ and $\mathbf H_{3}$ hold and $\sup_{\xi\in\Omega}\mu(\xi)\le\mu_{0}$ for some positive $\mu_{0}$. If
there exist measurable positive definite diagonal matrix-valued function $W(\cdot)$ and $V(\cdot)$,
and non-random positive $\nu(t)\in C^1([\tau_b,+\infty))$ and positive constants
$\alpha_{\nu},\beta_{\nu}$ satisfying
\begin{align*}
\lim_{t\rightarrow+\infty}\nu&(t)=+\infty,
   \quad \sup_{t\geq0}\frac{\nu^{\prime}(t)}{\nu(t)}\leq\alpha_{\nu}, \quad
   \sup_{t\geq0}\frac{\nu(t)}{\nu(t-\tau(t))}\leq\beta_{\nu},
\end{align*}
and positive constants $\rho_{1}$, $\kappa'$ and $\kappa$ $(\kappa>\kappa')$ such that $P(\xi)\le \rho_{1}P(\xi')$ holds
for all $\xi,\xi'\in\Omega$, and $M_{k}\le-\kappa P(\xi^{k})$, $N_{k}\le(\kappa'/\rho_{1})P(\xi^{k})$ hold, where
\begin{align*}
M_{k}=&\alpha_{\nu} P(\xi^{k})-2P(\xi^{k})D(\xi^{k})+a(\xi^{k})\|P(\xi^{k})\|_{2}G^2\\
&+P(\xi^{k})A(\xi^{k})V(\xi^k)^{-1}A^{\top}(\xi^{k})P^{\top}(\xi^{k})+GV(\xi^{k})G\\
&+P(\xi^{k})B(\xi^{k})W(\xi^{k})^{-1}B^{\top}(\xi^{k})P^{\top}(\xi^{k})\\
&+\mu(\xi^{k})\bigg\{\E[P(\xi^{k+1})|\F_{k}]-P(\xi^{k})\bigg\},\\
N_{k}=&\beta_{\nu}[GW(\xi^{k})G+a(\xi^{k})\|P(\xi^{k})\|_{2}G^{2}],\\
G=&\text{diag}[G_{1},\cdots,G_{n}],
\end{align*}
where the matrix norm $\Vert\cdot\Vert_2$ is induced by vector Euclidean norm,
then the system (\ref{nn}) is $\nu$-stable in probability.
\end{theorem}
\begin{proof}
The existence and uniqueness of the solution of the system (\ref{nn}) is also guaranteed by Theorem \ref{ext}. Consider the auxiliary functional 
$V_{2}(x_t,r(t),t)=\nu(t)x^{\top}(t)P(r(t))x(t)$, where $r(\theta)=r(0)$ for $-\tau_b\leq\theta <0$.
Letting $r(t)=\xi^{k}$ and denoting $\sigma_{t}=\sigma(g(x(t)),g(x(t-\tau(t))),r(t))$, we obtain the infinitesimal generator of $V_{2}(x_t,r(t),t)$ for $t\geq 0$ as follows:
\begin{align*}
   &\hat{\mathcal{A}}V_2(x_t,r(t),t)\leq\nu(t)
       x^{\top}(t)\Big[
           \frac{\nu^{\prime}(t)}{\nu(t)}P(\xi^k)-2P(\xi^k)D(\xi^k)\\
           &+\mu(\xi^k)
           \left\{
               \mathbb{E}[P(\xi^{k+1})|\mathcal{F}_k]-P(\xi^k)
           \right\}
           +a(\xi^k)\Vert P(\xi^k)\Vert_2G^2
       \Big]x(t)
   \\
   &+\nu(t)a(\xi^k)\Vert P(\xi^k)\Vert_2x^{\top}(t-\tau(t))G^2x(t-\tau(t))
   +2\nu(t)x^{\top}(t)\\
   &\times \Big[
   P(\xi^k)A(\xi^k)g(x(t))
   +P(\xi^k)B(\xi^k)g(x(t-\tau(t)))
   \Big].
\end{align*}

For positive definite diagonal matrix $V(\xi^{k})$, it follows that
\begin{align*}
&2x^{\top}(t)\left[P(\xi^{k})A(\xi^{k})\right]g(x(t))\\
=&2x^{\top}(t)\left[P(\xi^{k})A(\xi^{k})V(\xi^{k})^{-1/2}V(\xi^{k})^{1/2}\right]g(x(t))\\
\le& x^{\top}(t)\left[P(\xi^{k})A(\xi^{k})V(\xi^{k})^{-1}A^{\top}(\xi^{k})P^{\top}(\xi^{k})\right]x(t)
+g^{\top}(x(t))V(\xi^{k})g(x(t))\\
\le& x^{\top}(t)\Big[P(\xi^{k})A(\xi^{k})V(\xi^{k})^{-1}A^{\top}(\xi^{k})P^{\top}(\xi^{k})
+GV(\xi^{k})G\Big]x(t).
\end{align*}

In the same manner, for the positive definite diagonal matrix $W(\xi^{k})$, we have
\begin{align*}
&2x^{\top}(t)\left[P(\xi^{k})B(\xi^{k})\right]g(x(t-\tau))\\
\le& x^{\top}(t)\left[P(\xi^{k})B(\xi^{k})W(\xi^{k})^{-1}B^{\top}(\xi^{k})P^{\top}(\xi^{k})\right]x(t)\\
&+x^{\top}(t-\tau(t))GW(\xi^{k})Gx(t-\tau(t)).
\end{align*}
Since $
\sup_{t\geq0}\frac{\nu^{\prime}(t)}{\nu(t)}\leq\alpha_{\nu}$ and
$\sup_{t\geq0}\frac{\nu(t)}{\nu(t-\tau(t))}\leq\beta_{\nu}$, we have
\begin{align*}
&\hat{\mathcal{A}}V_2(x_t,r(t),t)\\
\leq& \nu(t)x^{\top}(t) M_k x(t)
+\nu(t-\tau(t)) x^{\top}(t-\tau(t)) N_k x(t-\tau(t))\\
\leq&
\frac{\kappa^{\prime}}{\rho_1}\nu(t-\tau(t)) x^{\top}(t-\tau(t))P(\xi^k) x(t-\tau(t))-\kappa V_2(x_t,r(t),t)\\
\leq& \kappa^{\prime} V_2(x_{t-\tau(t)},r(t-\tau(t)),t-\tau(t))-\kappa V_2(x_t,r(t),t).
\end{align*}
Therefore, the proof of Theorem \ref{thmnn2} can be easily accomplished from Theorem \ref{Hal}.
\end{proof}

\begin{remark}\label{remark:4}
From the proof of Theorems \ref{thmnn1} and \ref{thmnn2},
it is evident to see that in the case of bounded delay, we can choose appropriate $\alpha>0$ and $\nu(t)=\exp(\alpha t)$;
in the case of unbounded delay, when $|\tau^{\prime}(t)|\leq u$, $0<u<1$, we can choose appropriate $\alpha>0$ and $\nu(t)=(t+\tau_b+1)^{\alpha}$;
when $\tau(t)\leq t - t/\ln t$ for sufficient large $t$, we can choose $\nu(t)=\ln(t+\tau_b+1)$;
when $\tau(t)\leq t - t^{\alpha}$ for sufficient large $t$, $0<\alpha<1$, we can choose $\nu(t)=\ln\ln(t+\tau_b+3)$.
These results serve as a stochastic version of the $\mu$-stability in \cite{chen2017}.
\end{remark}

\section{Illustrative Examples}\label{sec:examples}

In this section, two illustrative examples are provided to demonstrate the main results derived in this paper.

\begin{figure}[htpb]
   \begin{center}
   \subfigure[Unstable subsystem with r=0.]{\label{fig1}
   \includegraphics[width=0.5\textwidth]{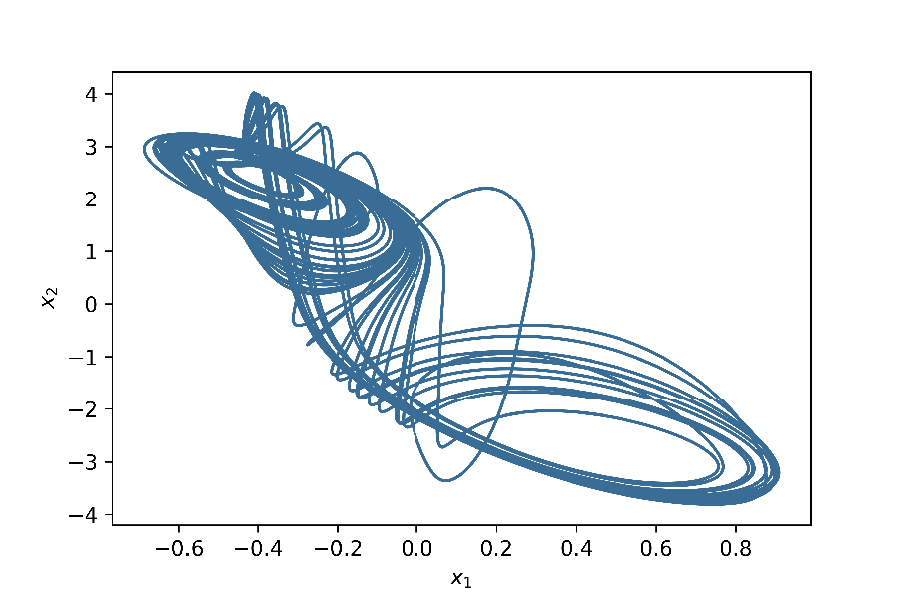}}
   \subfigure[Convergence dynamics with $\tau(t)\equiv 1$.]{\label{fig2}
   \includegraphics[width=0.5\textwidth]{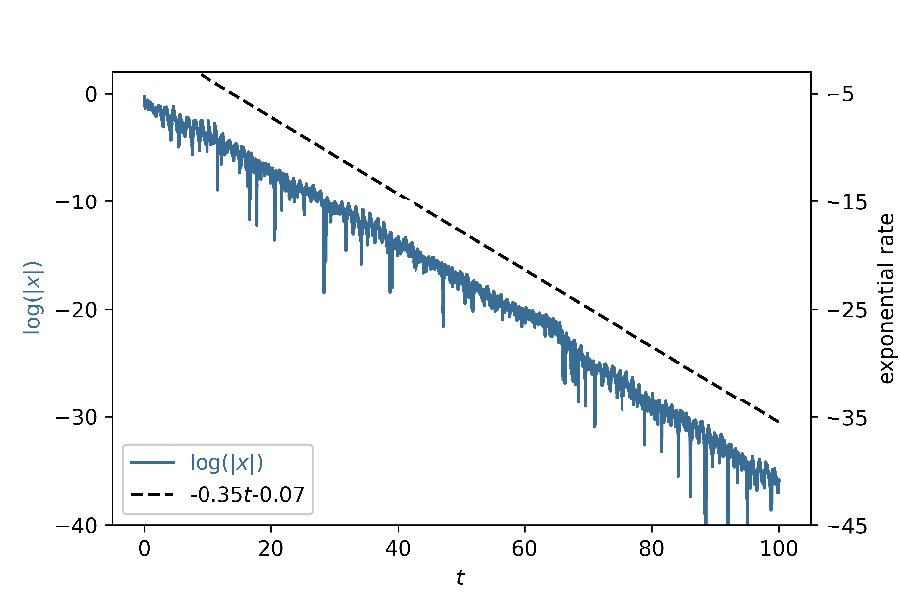}}
   \subfigure[Convergence dynamics with $\tau(t)= 0.1t+1$.]{\label{fig3}
   \includegraphics[width=0.5\textwidth]{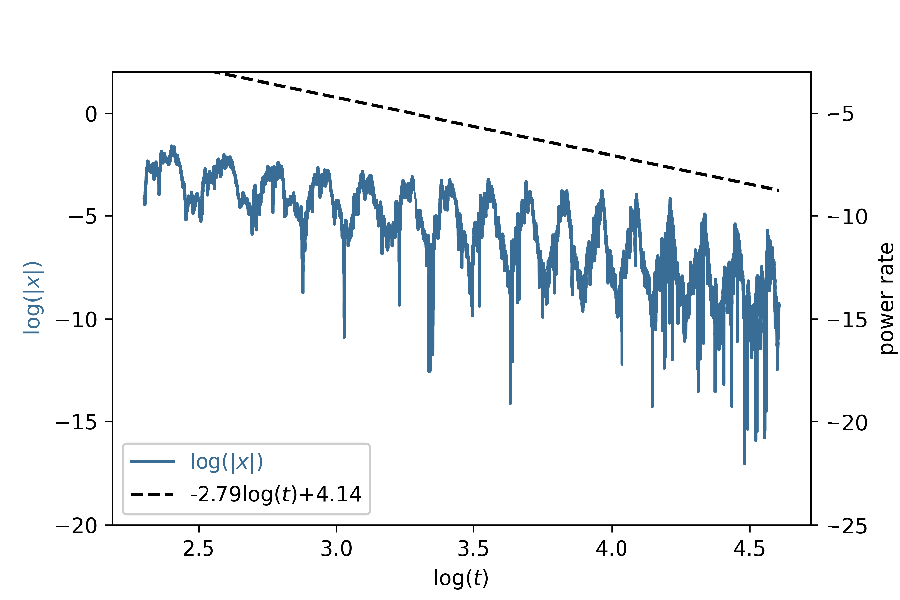}}
   \end{center}
   \caption{
   Panel (a): Attractor dynamic behaviors of the subsystem of (\ref{nn}) with r=0; the curve is plotted by disregarding the initial time interval $[0,500]$ from the whole interval $[0,1000]$.
   Panel (b): Convergent dynamics behaviors of (\ref{nn}) with $\tau(t)\equiv 1$ and initial value $[-0.4,0.6]$;
   the curve is plotted with time interval $[0,100]$.
   Panel (c): Convergent dynamics behaviors of (\ref{nn}) with $\tau(t)= 0.1t+1$ and initial value $[-0.4,0.6]$;
   the curve is plotted by disregarding the initial time interval $[0,10]$ from the whole interval $[0,100]$.
   }
\end{figure}

To begin with, consider a switching delayed networked system (\ref{nn}) consisting of two 2-dimension subsystems. The switching signal $r(t)$ takes values in $\{0,1\}$. The parameters of subsystem associated with $r(t)=0$ are given by
\begin{align*}
B(0)=\left[\begin{array}{cc} -1.5&   -0.1\\
   -0.2&   -2.5\end{array}\right], \; D(0)=I_{2}, \; A(0)=\left[\begin{array}{cc} 2 &  -0.1\\
   -5&    3\end{array}\right],
\end{align*}
with $\sigma(g(x(t)),g(x(t-\tau(t))),0)=g(x(t-\tau(t)))$, and those parameters associated with $r(t)=1$ are set as
\begin{align*}
B(1)=\left[\begin{array}{ll} 3.74&    2.5345\\
   -0.2280  &  5.7981\end{array}\right], \;
D(1)=8.0188I_{2}, \; A(1)=0_{2},
\end{align*}
with $\sigma(g(x(t)),g(x(t-\tau(t))),1)=g(x(t-\tau(t)))$.
The nonlinear output functions $g_{1}(\rho)$ and $g_{2}(\rho)$ are taken as $g_{1}(\rho)=g_{2}(\rho)=\tanh(\rho)$ which satisfy Hypothesis $\mathbf H_{2}$ with $G_{1}=G_{2}=1$. The delay function is chosen as $\tau(t)\equiv 1$.

It has been shown in \cite{Lu} that the subsystem associated with $r(t)=0$ has a chaotic attractor if the noisy term is removed. However, it is shown in Fig. \ref{fig1} that this subsystem is unstable at the origin. By contrast, the subsystem associated with $r(t)=1$ can be validated to be asymptotically stable at the origin by applying Theorem 4 in the non-random case. In the simulation, the solution of (\ref{nn}) is numerically plotted by the Euler method with a time step $0.001$.

The switching signal $r(t)$ is constructed by a switching time sequence following a mixed Poisson process with $\mu(0)=50$ and $\mu(1)=1$ and a discrete-time process $\xi^{k}$. The discrete-time process $\xi^{k}$ is described as follows. Let $X_{k}$ be an identical independent process taking values in $\{-1,1\}$ with equal probability. Setting $S_{n}=\sum_{k=1}^{n}X_{k}$ and $Y_{n}=\max_{k\le n}S_{k}$, then $\xi_{k}$ can be defined by $\xi_{k}=0$ ($\xi^{k}=1$) if $Y_{k}=S_{k}$ ($Y_{k}>S_{k}$). It can be seen that $\xi_{k}$ is not a Markov chain since it depends on all history states of $X_{k}$.

In order to solve the conditions in Theorem \ref{thmnn1}, we set $\nu(t)=\exp(0.01*t)$, and thus $\alpha_\nu=0.01$ and $\beta_\nu=\exp(-0.01)$.
Moreover, we impose $P(0)\ge P(1)$. The term $\chi_{k}=\mu(\xi^{k})\{\E[P(\xi^{k+1})|\F_{k}]-P(\xi^{k})\}$ where the non-Markovian switching $r(t)$ defined above can be estimated by $P(0)\ge P(1)$. Specifically, for $\xi_{k}=0$, i.e., $S_{k}=\max_{l\le k}S_{l}$, 1) if $X_{k+1}=-1$ with probability $1/2$, we have $S_{k+1}<S_{k}=Y_{k}=Y_{k+1}$, which implies $\xi_{k+1}=1$; and 2) if $X_{k+1}=1$ with probability $1/2$, we have $S_{k+1}=Y_{k+1}$, which implies $\xi_{k+1}=0$. Hence, in the case of $\xi_{k}=0$, it can be obtained that $\chi_{k}=(1/2)\mu(0)[P(1)-P(0)]$. For the case of $\xi_{k}=1$, it is seen that if $X_{k+1}=-1$ with probability $1/2$, then $S_{k+1}<S_{k}<Y_{k}=Y_{k+1}$, which implies $\xi_{k+1}=1$. However, the value of $\xi_{k+1}$ can't be obtained with probability $1/2$. Fortunately, since $P(0)\ge P(1)$ is imposed, we can obtain $\chi_{k}\le (1/2)\mu(1)[P(0)-P(1)]$, which can be used to verify the matrix inequality condition in the presented theorems.

By the LMI toolbox in Matlab, we obtain
\begin{align*}
   Z=&\text{diag}[3.2867,3.2867],\;
   R(0)=\text{diag}[174.1615,174.1615],\\
   R(1)=&\text{diag}[32.1933,32.1933],\;
   Q=\left[\begin{array}{ll}  44.7951 &3.4509\\
      3.4509 &62.8182\end{array}\right],\\
   P(0)=&\left[\begin{array}{ll}  14.3049 &-0.0796\\
      -0.0796 &15.7607\end{array}\right],\;
   P(1)=\left[\begin{array}{ll}  5.1113 &0.5240\\
      0.5240 &3.5149\end{array}\right].
\end{align*}

The largest eigenvalue of $\Pi$ equals $-0.8913$.
Therefore, the switching delayed networked system (\ref{nn}) converges to the origin and the simulation result is shown in Fig. \ref{fig2}. Moreover, we set $\tau(t)=0.1t+1$.
Without modifying other experimental settings,
let $\nu(t)=(t+2)^{0.01}$,
and thus $\alpha_{\nu}=0.005$ and $\beta_{\nu}=0.89$.
By the same fashion, we obtain
\begin{align*}
   Z=&\text{diag}[19.6882,19.6882],\;
   R(0)=\text{diag}[872.3114,872.3114],\\
   R(1)=&\text{diag}[105.4197, 105.4197],\;
   Q=\left[\begin{array}{ll}  190.2480 &0.5240\\
      0.5240 &294.6908\end{array}\right],\\
   P(0)=&\left[\begin{array}{ll}  57.1906 &-2.2500\\
      -2.2500 &69.8648\end{array}\right],
   P(1)=\left[\begin{array}{ll}  14.5554 &1.0659\\
      1.0659 &12.9256\end{array}\right].
\end{align*}

So, the largest eigenvalue of $\Pi$ equals -1.8069.
Therefore, the switching delayed networked system (\ref{nn}) converges to the origin and the simulation result is shown in Fig. \ref{fig3}.

\section{Conclusions}
In this paper, we have presented a novel model of time-varying delayed switching stochastic systems with the switching signal modeled as a combination of Cox's process and a discrete-time adaptive process. This switching process is rather general as it includes the independent processes, Markovian jump and hidden Markov processes as special cases.
The existence and uniqueness of the solutions to the time-varying delayed system have been proved under the Lipschitz continuous and linear growth conditions. Based on the generalized Dynkin's formula for the right-continuous adaptive process, the Lyapunov functional and Halanay inequality approaches have been developed to analyze the stability issue of stochastically switching systems. As an application, we discussed sufficient conditions for the stability of switching delayed networked system with stochastic perturbations, which can be regarded as a stochastic extension of the delayed neural networks, from both theoretical and numerical ways. We highlight that, under our established conditions, the concerned system could be stable even if some of its subsystems are unstable or the switching is non-Markovian.


\begin{thebibliography}{99}

\bibitem{Antonyuk-JM-20}
S. V. Antonyuk, M. F. Byrka, M. Y. Gorbatenko, T. O. Lukashiv, and I. V. Malyk, ``Optimal control of stochastic dynamic systems of a random structure with Poisson switches and Markov switching," {\it Journal of Mathematics}, vol. 2020, pp. 1--9, 2020.

\bibitem{BhatVN-JQRS-94}
V. N. Bhat, ``Renewal approximations of the switched Poisson processes and their applications to queueing systems," {\it Journal of the Operational Research Society}, vol. 45, no. 3, pp. 345--353, 1994.

\bibitem{Bra}
M. S. Branicky, ``Multiple Lyapunov functions and other analysis tools for switched and hybrid systems," {\it IEEE Transactions on Automatic Control}, vol. 43, no. 4, pp. 475--482, 1998.

\bibitem{CL}
C. G. Cassandras and J. Lygeros, {\it Stochastic hybrid systems}. Boca Raton: CRC Press, 2006.

\bibitem{CKH}
A. Cetinkaya, K. Kashima, and T. Hayakawa, ``Stability and stabilization of switching stochastic differential equations subject to probabilistic state jumps," in {\it Proceedings of the 49th IEEE Conference on Decision and Control, CDC 2010, December 15-17, 2010, Atlanta, Georgia, USA}, IEEE, 2010, pp. 2378--2383.

\bibitem{CL1}
D. Chatterjee and D. Liberzon, ``On stability of randomly switched nonlinear systems," {\it IEEE Transactions on Automatic Control}, vol. 52, no. 12, pp. 2390--2394, 2007.

\bibitem{chen2017}
T. Chen and X. Liu, ``$\mu$ -stability of nonlinear positive systems with nnbounded time-varying delays," {\it IEEE Transactions on Neural Networks and Learning Systems}, vol. 28, no. 7, pp. 1710--1715, 2017.

\bibitem{ChenT}
T. Chen and L. Wang, ``Global $\mu$ -stability of delayed neural networks with nnbounded time-varying delays," {\it IEEE Transactions on Neural Networks}, vol. 18, no. 6, pp. 1836--1840, 2007.

\bibitem{Costa}
O. L. V. Costa, M. D. Fragoso, and M. G. Todorov, {\it Continuous-Time Markov Jump Linear Systems}. in Probability and Its Applications. Berlin, Heidelberg: Springer, 2013.

\bibitem{Cox}
D. R. Cox, ``Some statistical methods connected with series of events," {\it Journal of the Royal Statistical Society: Series B (Methodological)}, vol. 17, no. 2, pp. 129--157, 1955.

\bibitem{Dav}
M. H. A. Davis, {\it Markov Models \& Optimization}. New York: Routledge, 2017.

\bibitem{Dynkin}
E. B. Dynkin, {\it Markov Processes}. Berlin, Heidelberg: Springer, 1965.

\bibitem{GengH-TSMCS-20}
H. Geng, Y. Liang, and Y. Cheng, ``Target state and Markovian jump ionospheric height bias estimation for OTHR tracking systems," {\it IEEE Transactions on Systems, Man, and Cybernetics: Systems}, vol. 50, no. 7, pp. 2599--2611, 2020.

\bibitem{GengH-IS-17}
H. Geng, Z. Wang, Y. Liang, Y. Cheng, and F. E. Alsaadi, ``State estimation for asynchronous sensor systems with Markov jumps and multiplicative noises," {\it Information Sciences}, vol. 417, pp. 1--19, 2017.

\bibitem{new_cox}
J. F. C. Kingman, {\it Poisson Processes}. Vol. 3. Clarendon Press, 1992.

\bibitem{Koz}
F. Kozin, ``A survey of stability of stochastic systems," {\it Automatica}, vol. 5, no. 1, pp. 95--112, 1969.



\bibitem{LLC}
B. Liu, W. Lu, and T. Chen, ``Generalized Halanay inequalities and their applications to neural networks with unbounded time-varying delays," {\it IEEE Transactions on Neural Networks}, vol. 22, no. 9, pp. 1508--1513, 2011.

\bibitem{LongY-IJSS-14}
Y. Long and G.-H. Yang, ``Fault detection for a class of nonlinear stochastic systems with Markovian switching and mixed time-delays," {\it International Journal of Systems Science}, vol. 45, no. 3, pp. 215--231, 2014.

\bibitem{Lu}
H. Lu, ``Chaotic attractors in delayed neural networks," {\it Physics Letters A}, vol. 298, no. 2, pp. 109--116, 2002.

\bibitem{Mao}
X. Mao, {\it Exponential Stability of Stochastic Differential Equations}, 1st edition. New York: CRC Press, 1994.

\bibitem{Mao3}
X. Mao, ``Exponential stability of stochastic delay interval systems with Markovian switching," {\it IEEE Transactions on Automatic Control}, vol. 47, no. 10, pp. 1604--1612, 2002.

\bibitem{Mao1}
X. Mao, A. Matasov, and A. B. Piunovskiy, ``Stochastic differential delay equations with Markovian switching," {\it Bernoulli}, pp. 73--90, 2000.

\bibitem{Mao2}
X. Mao and C. Yuan, {\it Stochastic differential equations with Markovian switching}. London: Imperial College Press, 2006.

\bibitem{MT}
S. P. Meyn and R. L. Tweedie, ``Stability of Markovian processes III: Foster--Lyapunov criteria for continuous-time processes," {\it Advances in Applied Probability}, vol. 25, no. 3, pp. 518--548, 1993.

\bibitem{Moh}
S.-E. A. Mohammed, ``Stochastic differential systems with memory: Theory, examples and applications," in {\it Stochastic Analysis and Related Topics VI}, Boston, MA: Birkh\"{a}user Boston, 1998, pp. 1--77.

\bibitem{Oks}
B. {\O}ksendal, {\it Stochastic Differential Equations}. in Universitext. Berlin, Heidelberg: Springer, 2003.

\bibitem{Protter}
P. E. Protter, {\it Stochastic Integration and Differential Equations}, vol. 21. in Stochastic Modelling and Applied Probability, vol. 21. Berlin, Heidelberg: Springer, 2005.

\bibitem{Sko}
A. V. Skorokhod, {\it Asymptotic methods in the theory of stochastic differential equations}. in Translations of mathematical monographs, no. v. 78. Providence, R.I: American Mathematical Society, 1989.

\bibitem{YuanC-Auto-04}
H. Yang, Z. Wang, Y. Shen, and F. E. Alsaadi, ``Self-triggered filter design for a class of nonlinear stochastic systems with Markovian jumping parameters," {\it Nonlinear Analysis: Hybrid Systems}, vol. 40, p. 101022, May 2021.



\bibitem{new_WFZ} 
Ligang Wu, Zhiguang Feng, and Wei Xing Zheng, ``Exponential Stability Analysis for Delayed Neural Networks With Switching Parameters: Average Dwell Time Approach,"  {\it IEEE Transactions on Neural Networks}, vol. 21, no. 9, pp. 1396--1407, Sep. 2010.

\bibitem{new_WSSC}
Z.-G. Wu, P. Shi, H. Su, and J. Chu, ``Stochastic Synchronization of Markovian Jump Neural Networks With Time-Varying Delay Using Sampled Data," {\it IEEE Transactions on Cybernetics}, vol. 43, no. 6, pp. 1796--1806, Dec. 2013.

\bibitem{YZZ}
G. G. Yin, B. Zhang, and C. Zhu, ``Practical stability and instability of regime-switching diffusions," {\it Journal of Control Theory and Applications}, vol. 6, no. 2, pp. 105--114, 2008.

\bibitem{YL}
C. Yuan and J. Lygeros, ``On the exponential stability of switching diffusion processes," {\it IEEE Transactions on Automatic Control}, vol. 50, no. 9, pp. 1422--1426, 2005.

\bibitem{Mao4}
C. Yuan and X. Mao, ``Robust stability and controllability of stochastic differential delay equations with Markovian switching," {\it Automatica}, vol. 40, no. 3, pp. 343--354, 2004.

\end{thebibliography}
\end{document}